\documentclass[12pt,reqno]{amsart}
 
\usepackage{amsmath,amsthm,amssymb,mathrsfs,amscd,epsfig,color}
\usepackage{url}
\usepackage{graphicx}
\usepackage{mathrsfs}
\usepackage{ulem}
\usepackage{fancybox}
\usepackage{pb-diagram} 
\usepackage[all]{xy}
\usepackage{comment}
\usepackage{mathtools}
\usepackage{centernot}
\usepackage{placeins} % for \FloatBarrier 
\usepackage{url}

\makeatletter

\makeatletter
    \newcommand\contFrac{\@ifstar{\@contFracStar}{\@contFracNoStar}}

   \def\singleContFrac#1#2{%
        \begin{array}{@{}c@{}}%
            \multicolumn{1}{c|}{#1}%
            \\%
            \hline%
           \multicolumn{1}{|c}{#2}%
        \end{array}%
   }

    \def\@contFracNoStar#1{%
        \mathchoice{% * Display style
            \@contFracNoStarDisplay@#1//\@nil%
        }{%           * Text style
            \@contFracNoStarInline@#1//\@nil%
        }{%           * Script style
            \@contFracNoStarInline@#1//\@nil%
        }{%           * Script script style
            \@contFracNoStarInline@#1//\@nil%
        }%
    }

    \def\@contFracNoStarDisplay@#1//#2\@nil{%
        \@ifmtarg{#2}{%
            #1%
        }{%
            #1+\cfrac{1}{\@contFracNoStarDisplay@#2\@nil}%
        }%
    }

        \def\@contFracNoStarInline@#1//#2\@nil{%
            \@ifmtarg{#2}{%
                #1%
            }{%
                #1 \@@contFracNoStarInline@@#2\@nil%
            }%
        }
        \def\@@contFracNoStarInline@@#1//#2\@nil{%
            \@ifmtarg{#2}{%
                + \singleContFrac{1}{#1}%
            }{%
                + \singleContFrac{1}{#1} \@@contFracNoStarInline@@#2\@nil%
            }%
        }

    \def\@contFracStar#1{%
        \mathchoice{% * Display style
            \@contFracStarDisplay@#1////\@nil%
        }{%           * Text style
            \@contFracStarInline@#1//\@nil%
        }{%           * Script style
            \@contFracStarInline@#1//\@nil%
        }{%           * Script script style
            \@contFracStarInline@#1//\@nil%
        }%
    }

    \def\@contFracStarDisplay@#1//#2//#3\@nil{%
        \@ifmtarg{#2}{%
            #1%
        }{%
            #1 + \cfrac{#2}{\@contFracStarDisplay@#3\@nil}%
        }%
    }

        \def\@contFracStarInline@#1//#2\@nil{%
            \@ifmtarg{#2}{%
                #1%
            }{%
                #1 \@@contFracStarInline@@#2\@nil%
            }%
        }
        \def\@@contFracStarInline@@#1//#2//#3\@nil{%
            \@ifmtarg{#3}{%
                + \singleContFrac{#1}{#2}%
            }{%
                + \singleContFrac{#1}{#2} \@@contFracStarInline@@#3\@nil%
            }%
        }
\makeatother

\numberwithin{equation}{section}
\theoremstyle{plain}

\newtheorem*{thmA}{Theorem A}
\newtheorem*{thmB}{Theorem B}

\newtheorem{thm}{Theorem}[section]
\newtheorem{lem}[thm]{Lemma}

\newtheorem{pro}[thm]{Proposition}
\theoremstyle{definition}
\newtheorem{df}[thm]{Definition}

\newtheorem{rem}[thm]{Remark}

\newtheorem*{prf*}{Proof}
\newtheorem*{cla*}{Claim}

\newtheorem*{pf*}{}
\newtheorem*{lem*}{LemmaA}
\newtheorem*{lm*}{LemmaB}
\makeatletter
\@namedef{subjclassname@2020}{%
  \textup{2020} Mathematics Subject Classification}
\makeatother
\title[]
{Exact dimensionality of projected measures \\ for expanding rational semigroups}
\author{Yuto Nakajima }

\address{Faculty of Science and Engineering, Doshisha University, Kyoto, 610-0394, JAPAN}
\email{yunakaji@mail.doshisha.ac.jp}

\subjclass[2020]{}
\thanks{{\it Keywords}: 
Exact dimensionality; Ledrappier--Young formula; projection entropy;  rational semigroups}
\begin{document}

\begin{abstract} 
We study the dimension theory of expanding rational semigroups. An expanding rational semigroup is naturally described by a skew product on the product of the symbolic space and the Riemann sphere. For an invariant Borel probability measure of this skew product, we consider its disintegration over a symbolic factor and study the push-forwards of the conditional measures under the second coordinate projection.
We prove a Ledrappier--Young type formula for the projected conditional measures. In particular, if the original invariant measure is ergodic, then these projected conditional measures are exact dimensional for almost every fiber. Applying this result to the trivial factor yields exact dimensionality of the projected invariant measure. This result can be viewed as a backward analogue of the result of [D.-J. Feng and H. Hu. \textit{Comm. Pure Appl. Math.} \textbf{62} (2009), no. 11, 1435--1500].
\end{abstract}

\maketitle

\section{Introduction}\label{intro}
An important problem in dynamical systems is to understand how fundamental invariants, such as entropy, Lyapunov exponents, and dimension, are related to each other. 
One way to approach this problem is to study the local scaling behaviour of invariant measures. 
Let $\mu$ be a Borel probability measure on a metric space $(X,d)$. We say that $\mu$ is {\it exact dimensional} if there exists a constant $D\geq 0$ such that
\begin{equation}
\label{exactdimension}
\lim_{r\to 0}\frac{\log \mu(B(x,r))}{\log r}=D
\quad\text{for }\mu\text{-a.e. }x\in X,
\end{equation}
where $B(x,r)=\{y\in X\colon d(x,y)<r\}.$
In this case, we write $\dim\mu=D$ and call this number the dimension of $\mu$. This dimension agrees with several other standard notions of dimension; see \cite[Section~10]{Fal} and \cite[Section~4]{Y82} .

Young~\cite{Y82} proved that every hyperbolic ergodic invariant measure of a smooth diffeomorphism on a compact Riemannian surface is exact dimensional, and related its dimension to entropy and Lyapunov exponents. Ledrappier and Young~\cite{LYI85,LYII85} further developed these ideas in higher dimensions.
In particular, they studied conditional measures along stable and unstable manifolds and proved their exact dimensionality. Later, Barreira, Pesin and Schmeling~\cite{BPS} showed that every hyperbolic ergodic invariant measure of a smooth diffeomorphism is exact dimensional.

Beyond the smooth setting, exact dimensionality has been investigated in many classes of non-smooth and non-invertible dynamical systems. A particularly important class is given by iterated function systems (IFSs), that is, finite collections of strictly contractive maps on complete metric spaces. In this direction, several developments have been obtained for conformal IFSs, random IFSs, self-affine IFSs, and Furstenberg measures; see \cite{BK17,Fen,FH,HS,MiUr16,Rap}.

Among these works, the result of Feng and Hu~\cite{FH} is especially relevant to this paper. They proved that the push-forward of an ergodic shift-invariant Borel probability measure on the symbolic space under the canonical projection map of an IFS is exact dimensional. A main difficulty is that the projection map is generally not one-to-one. Hence, the usual metric entropy of the shift map generally does not give the correct entropy term for the dimension of the projected measure. To deal with this loss of information under projection, Feng and Hu introduced the projection entropy (see \cite[Definition 2.1]{FH}).

The purpose of this paper is to extend the idea of Feng and Hu~\cite{FH} to expanding rational semigroups. A finitely generated rational semigroup is naturally described by a skew product, and the Julia set of the semigroup is obtained as the projection of the Julia set of this skew product. Projecting an invariant measure for the skew product gives a natural measure on the Julia set of the semigroup. Moreover, we consider the disintegration of an invariant measure over a symbolic factor and study the projected conditional measures.

As in the setting of IFSs, the projection is generally not one-to-one. Hence, the usual metric entropy of the skew product does not directly describe the local scaling behaviour of the projected measure.
Moreover, there is an additional difficulty in the rational semigroup setting.
Unlike the branch maps of an IFS, rational maps are not globally injective. To overcome this, we use partitions adapted to local inverse branches and define a corresponding projection entropy (see Definition~\ref{projent2}). Our main theorem gives a Ledrappier--Young type formula for the local dimensions of the projected conditional measures (see Theorem A). In particular, these projected conditional measures are exact dimensional for almost every fiber if the original invariant measure is ergodic.

\subsection{Rational semigroups}
The study of the dynamics of rational semigroups was initiated independently and around the same time by Hinkkanen and Martin~\cite{HM} and by Sumi~\cite{Sumi1}. We give a rigorous setting.

Let $\widehat{\mathbb C}$ denote the Riemann sphere. Let ${\rm Rat}$ denote  the set of all non-constant rational maps on $\widehat{\mathbb C}$ endowed with a metric $\overline{d}$ defined by 
\[\overline{d}(h_1, h_2)=\sup_{x\in \widehat{\mathbb C}}d_{\widehat{\mathbb C}}(h_1(x), h_2(x)),\]
where $d_{\widehat{\mathbb C}}$ is the spherical metric on $\widehat{\mathbb C}$. 
A rational semigroup is a subsemigroup of $\mathrm{Rat}$, where the semigroup operation is composition. 
If the generators are $f_1,f_2,\ldots$, we write $G=\langle f_1,f_2,\ldots\rangle.$
As in the case of a single rational map \cite{Mil}, the Julia set of a rational semigroup $G$ is defined by
\[J(G):=\left\{z\in\widehat{\mathbb C}\colon G \text{ is not normal on any neighbourhood of } z \right\}.\]
Let $m\in\mathbb N$, and let 
\[\Sigma_m:=\{1,\ldots,m\}^{\mathbb Z}.\] 
We denote by $\sigma:\Sigma_m\to\Sigma_m$ the shift map.
For $(f_1,\ldots,f_m)\in ({\rm Rat})^m,$ define the associated skew product
$\tilde f:\Sigma_m\times\widehat{\mathbb C}\to\Sigma_m\times\widehat{\mathbb C}$  by
\[\tilde f(\omega,z)= (\sigma\omega, f_{\omega_1}(z)) \ \text{for}\ 
\omega=(\ldots, \omega_{0},\omega_{1},\omega_2,\ldots)\in\Sigma_m.\]
For each $\omega\in\Sigma_m$, the fiberwise Julia set is defined by
\[J_\omega(\tilde f):=\left\{z\in\widehat{\mathbb C}\colon \{f_{\omega_n}\circ\cdots\circ f_{\omega_1}\}_{n\ge1} \text{ is not normal on any neighbourhood of }z \right\}.\]
The Julia set of the skew product is defined by
\[J(\tilde f):=\overline{\bigcup_{\omega\in\Sigma_m}\{\omega\}\times J_\omega(\tilde f)},\]
where the closure is taken in $\Sigma_m\times\widehat{\mathbb C}$. 
Let  $\Pi:\Sigma_m\times\widehat{\mathbb C}\to\widehat{\mathbb C}$ be the second coordinate projection. 
\begin{rem}
Let $G=\langle f_1,\ldots,f_m\rangle$ be a finitely generated rational semigroup, and let $\tilde f:\Sigma_m\times\widehat{\mathbb C} \to \Sigma_m\times\widehat{\mathbb C}$ be the skew product associated with $(f_1,\ldots,f_m).$
By \cite[Remark~2.3]{SU}, $J(\tilde f)$ is completely invariant under $\tilde f$, that is, 
\[\tilde f (J(\tilde f))=\tilde f^{-1} (J(\tilde f ))=J(\tilde f),\] and the second coordinate projection satisfies
\begin{equation}
\label{relation}
  \Pi(J(\tilde f))=J(G).
\end{equation}
Hence, $\Pi$ serves as the natural projection from the Julia set of the skew product to the Julia set of the rational semigroup.

By \cite[Lemma~0.2]{S3},
the Julia set $J(G)$ satisfies the backward self-similarity relation
\[ J(G)=\bigcup_{i=1}^m f_i^{-1}(J(G)),
\]
which is analogous to the self-similarity of limit sets of IFSs.

\end{rem}
For $z\in\widehat{\mathbb C}$, let $T_z\widehat{\mathbb C}$ be the complex tangent space at $z$. For a holomorphic map $\varphi$ on an open subset $V\subset \widehat{\mathbb C}$, we write
\[D\varphi_z:T_z\widehat{\mathbb C}\to T_{\varphi(z)}\widehat{\mathbb C}\]
for its derivative at $z$. The norm $\|\varphi'(z)\|$ denotes the operator norm of $D\varphi_z$ with respect to the spherical metric on $\widehat{\mathbb C}$. 

For $n\in\mathbb N$ and $(\omega, z)\in \Sigma_m\times \widehat{\mathbb C}$, define
\[\left\|(\tilde f^n)'(\omega,z)\right\|:=\left\|(f_{\omega_n}\circ\cdots\circ f_{\omega_1})'(z)\right\|.\]
\begin{df}\label{exprat}
Let $G=\langle f_1,\ldots,f_m\rangle$ be a finitely generated rational semigroup, and let $\tilde f:\Sigma_m\times\widehat{\mathbb C} \to \Sigma_m\times\widehat{\mathbb C}$ be the skew product associated with $(f_1,\ldots,f_m).$
We say that $G$ is expanding if $J(G)\neq\emptyset$ and if $\tilde f$ is expanding along the fibers of $J(\tilde f)$. That is, there exist constants $C>0$ and $\lambda>1$ such that for every $n\geq 1$, \[\inf_{(\omega,z)\in J(\tilde f)}\left\|(\tilde f^n)'(\omega,z)\right\|\geq C\lambda^n.\]
\end{df}
The dimension theory of expanding rational semigroups has been extensively developed. For example, under suitable separation conditions, Sumi~\cite{S1} proved that the Hausdorff dimension of the Julia set is equal to the unique zero of the pressure function associated with the semigroup. In the more complicated overlapping case, Sumi and Urba\'nski~\cite{SU} developed a transversality argument for parametrized families of expanding rational semigroups and obtained Bowen's formula for the Julia sets of such semigroups for almost every parameter value. Jaerisch and Sumi \cite{JS15} established the multifractal formalism for expanding rational semigroups.
\subsection{Measure-theoretic entropy}
Entropy will play a central role in our arguments. We recall some basic notation (see, e.g., \cite{Do, Mane, PU,Wal} for details).
Let $(\Omega,\mathcal F,\nu)$ be a probability space. For a sub-$\sigma$-algebra $\mathcal A\subset\mathcal F$ and $g\in L^1(\Omega,\mathcal F,\nu)$, we denote by $E_{\nu}(g\mid\mathcal A)$ the conditional expectation of $g$ with respect to $\mathcal A$.
For a countable measurable partition $\xi$ of $\Omega$, define the conditional information of $\xi$ given $\mathcal A$ by
\[I_{\nu}(\xi\mid\mathcal A):= -\sum_{A\in\xi} \mathbf 1_A\log E_{\nu}(\mathbf 1_A\mid\mathcal A),\]
where $\mathbf 1_A$ is the indicator function of $A.$
The conditional entropy of $\xi$ given $\mathcal A$ is defined by
\[H_{\nu}(\xi\mid\mathcal A):=\int_{\Omega} I_{\nu}(\xi\mid\mathcal A)\,d\nu.\]
%If $\mathcal A$ is the trivial $\sigma$-algebra $\{\emptyset, \Omega\}$, we simply write
%\[I_{\nu}(\xi) :=
%I_{\nu}(\xi\mid\{\emptyset, \Omega\}), \  H_{\nu}(\xi) :=H_{\nu}(\xi\mid\{\emptyset, \Omega\}).\]
For a topological space $\Delta,$ let $\mathcal B(\Delta)$ denote the Borel $\sigma$-algebra on $\Delta.$
Let 
\[\mathcal I:=\{B\in\mathcal B(J(\tilde f))\colon \tilde f^{-1}B=B\}\]
be the invariant $\sigma$-algebra on $J(\tilde f)$. We also write $\gamma:=\mathcal B(\widehat{\mathbb C}).$

\subsection{Statements of main results}
In this section, we state the main results of this paper. Let $G=\langle f_1,\ldots,f_m\rangle$ be an expanding rational semigroup, and let $\tilde f:J(\tilde f)\to J(\tilde f)$ be the skew product associated with $(f_1,\ldots,f_m).$ 

To treat projected invariant measures and projected conditional measures in a unified manner, we introduce a factor of the skew product. 
Let $p: J(\tilde f)\rightarrow \Sigma_m$ be a measurable map such that
\begin{equation}
\label{commute}
\sigma \circ p=p\circ \tilde f.
\end{equation}
For two sub-$\sigma$-algebras $\mathcal A_1$ and $\mathcal A_2$, we denote by $\mathcal A_1\vee\mathcal A_2$ the smallest $\sigma$-algebra containing both $\mathcal A_1$ and $\mathcal A_2$.
Following Feng and Hu~\cite{FH} and  Mihailescu and  Urba\'nski~\cite{MiUr16}, we introduce the notion of {\it conditional projection entropy}.

\begin{df}\label{projent2}
Let $\mathcal R$ be a finite Borel partition of $J(\tilde{f})$. Let $\mu$ be an $\tilde{f}$-invariant Borel probability measure on $J(\tilde{f})$. Define
\[h_{\Pi, p}(\tilde{f},\mu, \mathcal R):= H_\mu(\mathcal R\mid  p^{-1}\mathcal B({\Sigma_m})\vee \tilde{f}^{-1}\Pi^{-1}\gamma)-
H_\mu(\mathcal R\mid p^{-1}\mathcal B({\Sigma_m})\vee\Pi^{-1}\gamma).\]
This is called the {\it conditional projection entropy of $\mu$ under the projection $\Pi$ with respect to the partition $\mathcal R$ over $p$}.
For $x\in J(\tilde f)$ we also define 
\begin{align*}&h_{\Pi, p}(\tilde{f},\mu, x, \mathcal R)\\&:=
E_{\mu}\left(I_\mu(\mathcal R\mid p^{-1}\mathcal B({\Sigma_m})\vee\tilde{f}^{-1}\Pi^{-1}\gamma)-
I_\mu(\mathcal R\mid p^{-1}\mathcal B({\Sigma_m})\vee\Pi^{-1}\gamma)\mid \mathcal I\right)(x).\end{align*}
This is called the {\it conditional local projection entropy of $\mu$ at $x$ under the projection $\Pi$ with respect to the partition $\mathcal R$ over $p$}.\end{df}

Note that 
\[h_{\Pi, p}(\tilde{f},\mu, \mathcal R)=\int_{J(\tilde f)} h_{\Pi, p}(\tilde{f},\mu, x, \mathcal R)\ d\mu(x).\]
For $j\in \{1,\ldots, m\}$ set
\[[j]:=\{\omega=(\omega_i)_{i=-\infty}^{\infty}\in \Sigma_m\colon \omega_1=j\}.\]
In the case of conformal IFSs, we use the one-cylinder partition of the symbolic space (see \cite[Definition 2.1]{FH}). 
However, in the present setting a rational map is not globally injective in $\widehat{\mathbb C}$. 
Hence, we introduce condition (C$\Pi$) for a finite Borel partition $\mathcal R$ of $J(\tilde f)$ as follows.
\begin{itemize}
\item [(C$\Pi$):] For any $R\in \mathcal R$ there exist $j\in \{1,..., m\}$, a compact set $K_R\subset \widehat{\mathbb C},$ and an open set $U_R\subset \widehat{\mathbb C}$ such that 
\begin{itemize}
\item $R\subset J(\tilde f)\cap ([j]\times K_R);$ 
\item $K_R\subset U_R;$
\item $f_{j}$ is injective on $U_R.$
\end{itemize}
\end{itemize}
We call such a partition {\it $\Pi$-adapted} for $\tilde f$.
Intuitively, if a finite Borel partition $\mathcal R$ of $J(\tilde{f})$ is $\Pi$-adapted, then $\mathcal R$ records both the current symbol $\omega_1$ and a local inverse branch of the corresponding rational map. The existence of such partitions is a simple consequence of expansion and compactness (see Proposition~\ref{generate}). 

The following proposition shows that this definition is independent of the auxiliary partition satisfying (C$\Pi$).
\begin{pro}
\label{noproblem}
If $\mathcal R_1$ and $\mathcal R_2$ are $\Pi$-adapted, then for $\mu$-a.e. $x\in J(\tilde f),$
\[h_{\Pi, p}(\tilde f,\mu, x, \mathcal R_1)=h_{\Pi, p}(\tilde f,\mu, x, \mathcal R_2).\]
In particular, 
\[h_{\Pi, p}(\tilde f,\mu, \mathcal R_1)=h_{\Pi, p}(\tilde f,\mu, \mathcal R_2).\]
\end{pro}
Proposition~\ref{noproblem} shows that the conditional projection entropy is well defined independently of the choice of the auxiliary partition.
The proof of Proposition ~\ref{noproblem} is given in Section~3.2.

Let $h_{\Pi,p}(\tilde f,\mu)$ and $h_{\Pi,p}(\tilde f,\mu,x)$ denote the common values of the conditional projection entropy and the conditional local projection entropy over all $\Pi$-adapted partitions respectively.

For $x\in J(\tilde f)$ and $r>0$, define 
\[B^\Pi(x,r):=\{y\in J(\tilde f)\colon d_{\widehat{\mathbb C}}(\Pi y,\Pi x)<r\}.\]
For an $\tilde f$-invariant Borel probability measure $\mu$, let 
\[\mu=\int_{\Sigma_m} \mu_{\omega}\ d(p_{\ast}\mu)(\omega)\]
be the disintegration of $\mu$ over $p$ (see, e.g., \cite{Bo, Fur, Hoc} for details).
Our first main result is a Ledrappier--Young type formula for the projected conditional measures.

\begin{thmA}

Let $\mu$ be an $\tilde f$-invariant Borel probability measure on $J(\tilde f)$. Then for $p_\ast \mu$-a.e. $\omega\in \Sigma_m$, for $\mu_{\omega}$-a.e. $x\in p^{-1}(\{\omega\}),$
\[ \lim_{r\to 0}\frac{\log \mu_{\omega}(B^\Pi(x,r))}{\log r}=
\frac{h_{\Pi, p}(\tilde f,\mu, x)}{E_{\mu}(\log\|\tilde f'(x)\| \mid \mathcal I)}.\]
If $\mu$ is ergodic, then
for $p_\ast \mu$-a.e. $\omega\in \Sigma_m$,
for $\mu_{\omega}$-a.e. $x\in p^{-1}(\{\omega\}),$

\[\lim_{r\to 0}
\frac{\log \mu_{\omega}(B^\Pi(x,r))}{\log r}=
\frac{h_{\Pi, p}(\tilde f,\mu)}{\displaystyle\int_{J(\tilde f)} \log\|\tilde f'(x)\|\,d\mu(x)}.
\]
In particular, if $\mu$ is ergodic, then for $p_\ast \mu$-a.e. $\omega\in \Sigma_m$, the measure $\Pi_{\ast}\mu_{\omega}$ is exact dimensional. 
\end{thmA}
\begin{rem}
\begin{itemize}
\item In the case of a single expanding rational map, Ma\~n\'e~\cite{Mane2} proved that every ergodic invariant measure is exact dimensional. For random conformal expanding maps, Simmons and Urba\'nski~\cite{SiUr} proved that the conditional measures of an ergodic invariant measure are exact dimensional  for almost every fiber. In contrast, in the present setting we study projected conditional measures, for which the loss of information under projection has to be taken into account.

\item Mihailescu and Urba\'nski~\cite[Theorem~6.2]{MiUr20} also studied exact dimensionality of
projected conditional measures. Their setting, however, is different: they
consider equilibrium states under separation-type assumptions, while our
result applies to arbitrary invariant measures and accounts for overlaps
via conditional projection entropy.
\end{itemize}
\end{rem}
In the setting of IFSs, the push-forward of a shift-invariant measure on the symbolic space under the canonical coding map is usually regarded as an invariant measure for the IFS.
In the same spirit, if $\mu$ is a $\tilde f$-invariant measure, then it is natural to regard its projection $\Pi_{\ast}\mu$ as an invariant measure for the rational semigroup. By \eqref{relation}, the invariant measure for a rational semigroup is supported on the Julia set of the semigroup. 
Applying Theorem A to the trivial factor yields the following theorem. 
\begin{thmB}

Let $\mu$ be an $\tilde f$-invariant Borel probability measure on
$J(\tilde f)$. Then for $\mu$-a.e. $x\in J(\tilde f)$,
\[ \lim_{r\to 0}
\frac{\log \mu(B^\Pi(x,r))}{\log r}=
\frac{h_\Pi(\tilde f,\mu, x)}{E_{\mu}(\log\|\tilde f'(x)\| \mid \mathcal I)},\]
where 
\[h_\Pi(\tilde{f},\mu, x):=E_{\mu}\left(I_\mu(\mathcal R\mid \tilde{f}^{-1}\Pi^{-1}\gamma)-
I_\mu(\mathcal R\mid \Pi^{-1}\gamma)\mid \mathcal I\right)(x),\]
$\mathcal R$ is a $\Pi$-adapted partition.

If $\mu$ is ergodic, then for $\mu$-a.e. $x\in J(\tilde f)$,
\[
\lim_{r\to 0}\frac{\log \mu(B^\Pi(x,r))}{\log r}=
\frac{h_\Pi(\tilde f,\mu)}{\displaystyle\int_{J(\tilde f)} \log\|\tilde f'(x)\|\,d\mu(x)},
\]
where  \[h_\Pi(\tilde{f},\mu) :=H_\mu(\mathcal R\mid \tilde{f}^{-1}\Pi^{-1}\gamma)-H_\mu(\mathcal R\mid \Pi^{-1}\gamma),\]
$\mathcal R$ is a $\Pi$-adapted partition. In particular, if $\mu$ is ergodic, then the measure $\Pi_{\ast}\mu$ is exact dimensional. 

\end{thmB}

\subsection{Structure of the paper}
The rest of the paper is organized as follows. Throughout this paper, we mainly rely on the arguments of Feng and Hu~\cite{FH}, with additional ideas inspired by Mihailescu and  Urba\'nski~\cite{MiUr16}.
In Section~2, we collect the measure-theoretic and geometric ingredients needed for the proofs of the main results. 
In Section~3, we first construct a $\Pi$-adapted partition and prove that the conditional projection entropy is independent of the choice of the auxiliary $\Pi$-adapted partition. We then prove Theorem A, from which Theorem B follows by taking the trivial factor.

\section{Preliminaries}
This section summarizes the preliminary results needed in the sequel.
\subsection{A basic property of $\Pi$-adapted partitions}
In this section, we give a basic property of the $\Pi$-adapted partitions, which is the analogue of \cite[Lemma 4.7]{FH} for the one-cylinder partition of an IFS.
For a Borel partition $\mathcal R=\{R_1,\ldots, R_k\}$ of $J(\tilde f)$, let $\mathcal A(\mathcal R)$ denote the $\sigma$-algebra associated with $\mathcal R.$ Namely 
\[\mathcal A(\mathcal R)
=\left\{\bigcup_{i\in I}R_i \colon I\subset\{1,\ldots,k\}\right\}.\]
\begin{lem}
\label{simple}
Let $\tilde f$ be the associated skew product for $(f_1,\ldots,f_m)\in({\rm Rat})^m$.
Let $\mathcal R$ be a $\Pi$-adapted partition.
Then
\[\mathcal A(\mathcal R)\vee \tilde f^{-1}\Pi^{-1}\gamma =\mathcal A(\mathcal R)\vee \Pi^{-1}\gamma .\]
\end{lem}

\begin{proof}
We first prove
\begin{equation}\label{inclusion1}\mathcal A(\mathcal R)\vee \tilde f^{-1}\Pi^{-1}\gamma\subset \mathcal A(\mathcal R)\vee \Pi^{-1}\gamma .\end{equation}
It is enough to show that, for every $R\in\mathcal R$ and every $B\in\gamma$,
\[R\cap \tilde f^{-1}\Pi^{-1}B\in\mathcal A({\mathcal R})\vee \Pi^{-1}\gamma.\]
By (C$\Pi$), there exists $j=j(R)\in \{1,\ldots, m\}$ such that
$R\subset [j]\times \widehat{\mathbb C}$. Hence, for any $x=(\omega,z)\in R\cap \tilde f^{-1}\Pi^{-1}B$,
we have $\omega_1=j$, which implies that
\[\Pi(\tilde f x)=f_{j}(\Pi x).\]
Therefore, 
\[R\cap \tilde f^{-1}\Pi^{-1}B=R\cap \Pi^{-1}\bigl(f_{j}^{-1}(B)\bigr).\]
Since $f_{j}^{-1}(B)\in\gamma$, the right-hand side belongs to
$\mathcal A(\mathcal R)\vee \Pi^{-1}\gamma$. This implies \eqref{inclusion1}.

We next prove the reverse inclusion
\begin{equation}\label{inclusion2}\mathcal A(\mathcal R)\vee\Pi^{-1}\gamma \subset \mathcal A(\mathcal R)\vee \tilde f^{-1}\Pi^{-1}\gamma .\end{equation}
Let $R\in\mathcal R$. By (C$\Pi$), there exist $j=j(R)\in\{1,\ldots,m\}$ and an open set $U_R$ 
such that
\begin{equation}\label{good}R\subset J(\tilde f)\cap([j]\times U_R),\end{equation}
and $f_j$ is injective on $U_R$.
Let $C\in\gamma$.
We claim that
\[R\cap \Pi^{-1}C=R\cap \tilde f^{-1}\Pi^{-1}\bigl(f_j(C\cap U_R)\bigr).\]
Indeed, if $x=(\omega, z)\in R\cap \Pi^{-1}C$, then by \eqref{good}  we have $\omega_1=j$ and $\Pi x\in C\cap U_R$. Hence,
\[ \Pi(\tilde f x)=f_j(\Pi x)\in f_j(C\cap U_R),\]
which implies that
\[R\cap \Pi^{-1}C\subset R\cap \tilde f^{-1}\Pi^{-1}\bigl(f_j(C\cap U_R)\bigr).\]
Conversely, suppose that
\[x\in R\cap \tilde f^{-1}\Pi^{-1}\bigl(f_j(C\cap U_R)\bigr).\]
Then $\Pi x\in U_R$, and there exists $w\in C\cap U_R$ such that
\[f_j(\Pi x)=\Pi(\tilde f x)=f_j(w).\]
Since $\Pi x,w\in U_R$ and $f_j|_{U_R}$ is injective, we obtain
$\Pi x=w\in C$. Hence, $x\in R\cap \Pi^{-1}C$. This proves the claim.

Since $f_j|_{U_R}$ is injective and holomorphic, it is a homeomorphism from $U_R$ onto the open set $f_j(U_R)$. Hence,
\begin{equation*}
f_j(C\cap U_R)\in\gamma .\end{equation*}
Therefore,
\[R\cap \Pi^{-1}C\in\mathcal A(\mathcal R)\vee \tilde f^{-1}\Pi^{-1}\gamma ,\]
and \eqref{inclusion2} follows. Combining \eqref{inclusion1} and \eqref{inclusion2} completes the proof.
\end{proof}
\subsection{Conditional measures and density estimates}
In this section we give several density estimates used to relate conditional information to the local scaling of projected measures. 
Let $\tilde f: J(\tilde f)\to J(\tilde f)$ be the associated skew product for $(f_1,\ldots,f_m)\in({\rm Rat})^m$ and let $\mu$ be a Borel probability measure on $J(\tilde f)$. Since $J(\tilde f)$ is a compact metric space, $(J(\tilde f),\mathcal B(J(\tilde f)),\mu)$ is a Lebesgue probability space.
Let $p: J(\tilde f)\to \Sigma_m$ be a measurable map. The fibers of $p$ define a measurable partition
\[\eta_p=\{p^{-1}(\{y\})\colon y\in\Sigma_m\}.\]
Let
\[
\mu=\int_{\Sigma_m} \mu_y\,d(p_{\ast}\mu)(y)\]
be the disintegration of $\mu$ over $p$. Hence, $\mu_y$ is supported on
$p^{-1}(\{y\})$ for $p_{\ast}\mu$-a.e. $y$.
\begin{rem}
\label{disint}
For $x\in J(\tilde f)$, we write $\mu_x:=\mu_{p(x)}.$
Then for any measurable property $P(x)$, 
\[\text{for $\mu$-a.e. $x$, $P(x)$ holds }\] if and only if
\[\text{for $p_{\ast}\mu$-a.e. $y$, for $\mu_y$-a.e. $x\in p^{-1}\{y\}$, $P(x)$ holds }.\]\end{rem}

Let $\pi: J(\tilde f)\rightarrow \widehat{\mathbb C}$ be a measurable function. For $x\in J(\tilde f)$ and $r>0$, set
\[B^\pi(x,r):=\{y\in J(\tilde f)\colon d_{\widehat{\mathbb C}}(\pi(y),\pi(x))<r\}.\]
The following density lemma, which is a key ingredient in the proof of Theorem A, is a direct consequence of \cite[Proposition 3.5]{FH} applied to the partition $\eta_p$. Although \cite[Proposition 3.5]{FH} is stated for Euclidean balls, it also applies to spherical balls, since $\widehat{\mathbb C}$ is covered by two stereographic coordinate charts and the spherical metric is locally bi-Lipschitz equivalent to the Euclidean metric in these charts (see also \cite[Lemma~2.4]{Rap} for a statement in a more general setting).

\begin{lem}
\label{prodensity}
Let $\xi$ be a finite Borel partition of $J(\tilde f)$. For $x\in J(\tilde f)$, let $\xi(x)$ denote the element of $\xi$ containing $x$. Suppose that $\pi: J(\tilde f)\rightarrow \widehat{\mathbb C}$ is a measurable function. Then for $\mu$-a.e. $x\in J(\tilde f)$ we have
\[ \lim_{r\to0}\log\frac{\mu_x\bigl(B^\pi(x,r)\cap \mathcal \xi(x)\bigr)}{\mu_x\bigl(B^\pi(x,r)\bigr)}=
-I_\mu(\xi\mid  p^{-1}\mathcal B(\Sigma_m)\vee\pi^{-1}\gamma)(x).\]
Furthermore, set
\[h(x)=-\inf_{r>0}\log
\frac{\mu_{x}\bigl(B^\pi(x,r)\cap \mathcal \xi(x)\bigr) }{\mu_{x}\bigl(B^\pi(x,r)\bigr)}.\] 
Then $h\ge 0$ and $h\in L^1(J(\tilde f)).$
\end{lem}
The following density lemma is a direct consequence of \cite[Proposition 3.9]{FH}.

\begin{lem}
\label{prozero}
Assume that $\mu$ is $\tilde f$-invariant.
Let $\pi: J(\tilde f)\rightarrow \widehat{\mathbb C}$ be a bounded measurable function. Then for any $r>0$,
\[\lim_{n\to\infty}\frac{1}{n}\log \mu_{\tilde f^n x}\bigl(B^\pi(\tilde f^n  x,r)\bigr)= 0\]
for $\mu$-a.e. $x\in J(\tilde f) $.
\end{lem}
We also need the following result.
\begin{lem}\cite[Corollary 1.6, p. 96]{Mane}
\label{Manelem}

Let $T:X\to X$ be a measure-preserving transformation on a probability space $(X, \mathcal B, \nu)$. Let $F_k\in L^1(X, \nu)$ be a sequence that converges almost everywhere and in $L^1$ to $F\in L^1(X, \nu).$ Then \[\lim_{k\to \infty}\frac{1}{k}\sum_{j=0}^{k-1}F_{k-j}(T^j(x))=E_{\nu}(F \mid \mathcal I)(x)\]
almost everywhere and in $L^1.$
\end{lem}

Finally, we recall a standard fiberwise invariance property of conditional measures.
\begin{lem}
\label{fibinv}
Let $\mu$ be an $\tilde f$-invariant Borel probability measure on $J(\tilde f)$. Let $p:J(\tilde f)\rightarrow \Sigma_m$ be a measurable map such that
\[\sigma \circ p=p\circ \tilde f.\]
Let
\[\mu=\int_{\Sigma_m} \mu_\omega\,d(p_{\ast}\mu)(\omega)\]
be the disintegration of $\mu$ over $p$. Then
\[
\tilde f_{\ast}\mu_\omega=\mu_{\sigma\omega}
\ \text{for }p_{\ast}\mu\text{-a.e. }\omega\in \Sigma_m.\]
\end{lem}
For a proof, see, for example, \cite[Proposition~5.9]{Fur}.

\subsection{Distance estimates}
In this section we give some geometric estimates needed in the proof of the main results. We first recall a standard distortion estimate for conformal maps.
\begin{lem}\cite[Lemma 5.1]{FH}
\label{BDP}
Let $S: U\rightarrow S(U)\subset \mathbb R^d$ be a $C^1$ conformal diffeomorphism on an open set $U\subset \mathbb R^d,$ and $X$ a compact subset of $U.$ Let $c>1.$ Then there exists $r_0>0$ such that 
\[c^{-1}|S^{\prime}(x)||x-y|\le |S(x)-S(y)|\le c |S^{\prime}(x)||x-y|\]
for all $x\in X, y\in U$ with $|x-y|\le r_0.$
Here, $S^{\prime}(x)$ denotes the differential at $x,$ and $|S^{\prime}(x)|$ denotes the operator norm of $S^{\prime}(x).$
\end{lem}
For $(f_1,\ldots,f_m)\in({\rm Rat})^m,$ let $\tilde f$ be the associated skew product map.
Let $\mathcal R$ be a $\Pi$-adapted partition on $J(\tilde f)$. 
By condition (C$\Pi$), for any $R\in \mathcal R$ there exist $j=j(R)\in \{1,..., m\}$, a compact set $K_R\subset \widehat{\mathbb C},$ and an open set $U_R\subset \widehat{\mathbb C}$ such that $R\subset [j]\times K_R,$ $K_R\subset U_R,$ and $f_{j}$ is injective on $U_R.$ Let
\[\varphi_R=(f_j|_{U_R})^{-1}\colon f_j(U_R)\to U_R\]
be the corresponding inverse branch.
Since the inverse branch $\varphi_R$ is holomorphic and $f_j(K_R)$ is a compact subset of $f_j(U_R),$ applying Lemma~\ref{BDP} in local coordinates and using the local bi-Lipschitz equivalence between the Euclidean and spherical metrics, we obtain the same estimate with respect to the spherical metric: for any $c>1$ there exists $r_{R}>0$ such that 
\begin{equation}
\label{distRiemann}c^{-1}\|\varphi_R^{\prime}(\zeta_1)\|d_{\widehat{\mathbb C}}(\zeta_1, \zeta_2)\le 
d_{\widehat{\mathbb C}}\left(\varphi_R(\zeta_1),\varphi_R(\zeta_2)\right)\le 
c \|\varphi_R^{\prime}(\zeta_1)\|d_{\widehat{\mathbb C}}(\zeta_1, \zeta_2)\end{equation}
for all $\zeta_1\in f_j(K_R), \zeta_2\in f_j(U_R)$ with $d_{\widehat{\mathbb C}}(\zeta_1, \zeta_2)\le r_R.$ 

For $z\in \widehat{\mathbb C}$ and $r>0$ set
\[B_{\widehat{\mathbb C}}(z, r)=\{w\in \widehat{\mathbb C}\colon d_{\widehat{\mathbb C}}(z, w)<r\}.\]
Since $\mathcal R$ is finite and $f_{j(R)}(K_R)$ is contained in $f_{j(R)}(U_R)$ for each $R\in\mathcal R$, we may choose $r_0>0$ uniformly so that $B_{\widehat{\mathbb C}}(\zeta,r)\subset f_{j(R)}(U_R)$ whenever
$\zeta\in f_{j(R)}(K_R)$, $0<r<r_0$, and $R\in\mathcal R$.
Using this and \eqref{distRiemann} we obtain the following.
\begin{lem}
\label{BDP2}
Let $c>1.$ Then there exists $r_0>0$ such that for any $R\in \mathcal R, \zeta\in f_{j(R)}(K_R),$ and $0< r< r_0,$
\[B_{\widehat{\mathbb C}}(\varphi_R(\zeta), c^{-1}\|\varphi_R^{\prime}(\zeta)\|r)\subset \varphi_R(B_{\widehat{\mathbb C}}(\zeta, r))\subset B_{\widehat{\mathbb C}}(\varphi_R(\zeta), c\|\varphi_R^{\prime}(\zeta)\|r).\]
\end{lem}

For $x\in J(\tilde{f})$, let
$\mathcal R(x)$ denote the element of $\mathcal R$ containing $x$,
and put
\[\rho(x):=\|\tilde f'(x)\|^{-1}.\]
For $x\in J(\tilde{f})$ and $r>0$, define
\[B^{\Pi\circ\tilde f}(x,r)
 :=\{y\in J(\tilde f)\colon d_{\widehat{\mathbb C}}(\Pi(\tilde f y),\Pi(\tilde f x))<r\}.\]
We now apply Lemma~\ref{BDP2} to the skew product.
\begin{lem}
\label{Ballesti}
Let $c>1$. There exists $r_0>0$ such that for every $x\in J(\tilde f)$ and every
$0<r<r_0$,
\[ B^\Pi\bigl(x,c^{-1}\rho(x)r\bigr)\cap \mathcal R(x)
\subset B^{\Pi\circ\tilde f}(x,r)\cap \mathcal R(x)
\subset B^\Pi\bigl(x,c\rho(x)r\bigr)\cap \mathcal R(x).\]
\end{lem}

\begin{proof}
Take $x=(\omega,z)\in J(\tilde f)$, and put $R=\mathcal R(x)$. By condition (C$\Pi$), there exist $j\in \{1,..., m\}$, a compact set $K_R\subset \widehat{\mathbb C},$ and an open set $U_R\subset \widehat{\mathbb C}$ such that $R\subset [j]\times K_R,$ $K_R\subset U_R,$ and $f_{j}$ is injective on $U_R.$ Then
$z=\Pi x\in K_R, \omega_1=j,$
and hence
\[\Pi(\tilde f x)=f_j(z)\in f_j(K_R).\]
Set $\zeta:=\Pi(\tilde f x)=f_j(z).$
Since $\varphi_R(\zeta)=z=\Pi x$, we have
\[ \|\varphi_R'(\zeta)\|=\|f_j'(z)\|^{-1}=\|\tilde f'(x)\|^{-1}=\rho(x).\]
By this and Lemma~\ref{BDP2}, there exists $r_0>0$ such that for every
$0<r<r_0$, we have
\begin{equation}
\label{hougan2}
B_{\widehat{\mathbb C}}\bigl(\Pi x,c^{-1}\rho(x)r\bigr)
 \subset
\varphi_R(B_{\widehat{\mathbb C}}(\Pi(\tilde f x),r))
 \subset
B_{\widehat{\mathbb C}}\bigl(\Pi x,c\rho(x)r\bigr).
\end{equation}
We now claim that
\[
\Pi^{-1}\!\left(\varphi_R(B_{\widehat{\mathbb C}}(\Pi(\tilde f x),r))\right)\cap R=B^{\Pi\circ\tilde f}(x,r)\cap R .\]
Indeed, let $y=(\tau,u)\in R$. Then $\tau_1=j$, $u=\Pi y\in U_R$,
and hence
\[\Pi(\tilde f y)=f_j(u).\]
Thus
\[y\in B^{\Pi\circ\tilde f}(x,r)\cap R\ \text{
if and only if}\ 
f_j(u)\in B_{\widehat{\mathbb C}}(\Pi(\tilde f x),r)\ \text{and}\ y\in R.\]
Since $u\in U_R$ and $\varphi_R=(f_j|_{U_R})^{-1}$, this is equivalent to
\[
u=\varphi_R(f_j(u))\in\varphi_R(B_{\widehat{\mathbb C}}(\Pi(\tilde f x),r))\ \text{and}\ y\in R.\]
Equivalently,
\[
y\in\Pi^{-1}\!\left(\varphi_R(B_{\widehat{\mathbb C}}(\Pi(\tilde f x),r))\right)\cap R .\]
This proves the claim.
Combining the claim with \eqref{hougan2} yields
\[
B^\Pi\bigl(x,c^{-1}\rho(x)r\bigr)\cap R \subset B^{\Pi\circ\tilde f}(x,r)\cap R\subset B^\Pi\bigl(x,c\rho(x)r\bigr)\cap R ,\] as required.

\end{proof}

\section{Proofs}
In this section, we prove the remaining propositions and complete the proofs of the main results.
\subsection{Existence of $\Pi$-adapted partition}
The aim of this section is to prove the following.
\begin{pro}\label{generate}
There exists a $\Pi$-adapted partition  of $J(\tilde{f})$.
\end{pro}
\begin{proof}
For $j=1,\ldots,m$, define
\[J_j:=\Pi\bigl(J(\tilde{f})\cap \bigl([j]\times \widehat{\mathbb C}\bigr)\bigr)\subset J(G).\]
Since $\tilde{f}$ is expanding, we have $f_j'(z)\neq 0$ for any $ z\in J_j , j\in \{1,..., m\}.$
Hence, for each $j\in \{1,..., m\}$ and for every $z\in J_j$, there exist open neighborhoods $V_z$ and $U_z$ of $z$ such that 
\begin{itemize}
\item $f_j|_{U_z}$ is injective. 
\item $\overline{V_z}\subset U_z,$
where $\overline{V}$ is the closure of $V$ with respect to the spherical metric.
\end{itemize}Since $J_j$ is compact, there exist finitely many open sets 
$V_{j,1},\ldots,V_{j,N_j}$ such that
\[J_j\subset \bigcup_{\ell=1}^{N_j}V_{j,\ell},
\ \text{and}\
f_j|_{U_{j,\ell}}\ \text{is injective for each }\ell .\]
For each $j\in \{1,..., m\},$ set
\[E_{j, 1}:=J_j \cap V_{j, 1}, E_{j, \ell}:=J_j \cap \left(\bigcup_{k=1}^{\ell}V_{j,k}\setminus \bigcup_{k=1}^{\ell-1}V_{j,k}\right)\ \text{for all}\ \ell\in \{2,..., N_j\}.\]
  
Then we obtain a finite Borel partition
\[\mathcal E_j=\{E_{j,\ell}:1\leq \ell\leq N_j\}\]
of $J_j$ such that for any $\ell\in \{1,..., N_j\},$
\[
K_{j,\ell}:=\overline{E_{j,\ell}}\subset U_{j,\ell}\]
and $f_j$ is injective on $U_{j,\ell}$. Hence, 
\[\mathcal R:=
\{J(\tilde f)\cap([j]\times E_{j, \ell})\colon j\in \{1,..., m\}, \ell\in \{1,..., N_j\}\}\] 
satisfies (C$\Pi$), which is required.
\end{proof}

\subsection{Proof of Proposition~\ref{noproblem}}
Let $\mathcal R_1$ and $\mathcal R_2$ be $\Pi$-adapted. 
Since 
\[I_\mu(\mathcal R_1\vee \mathcal R_2\mid \mathcal A)=I_\mu(\mathcal R_1\mid \mathcal A)+I_\mu(\mathcal \mathcal R_2\mid \mathcal A(\mathcal R_1)\vee \mathcal A)\] holds for any sub-$\sigma$-algebra $\mathcal A$,
we have 
\begin{align*}
&I_\mu(\mathcal R_1\vee \mathcal R_2\mid p^{-1}\mathcal B({\Sigma_m})\vee\tilde{f}^{-1}\Pi^{-1}\gamma)-I_\mu(\mathcal R_1\vee \mathcal R_2\mid p^{-1}\mathcal B({\Sigma_m})\vee\Pi^{-1}\gamma)\\&=
I_\mu(\mathcal R_1\mid p^{-1}\mathcal B({\Sigma_m})\vee\tilde{f}^{-1}\Pi^{-1}\gamma)+I_\mu(\mathcal R_2\mid p^{-1}\mathcal B({\Sigma_m})\vee\mathcal A(\mathcal R_1)\vee \tilde{f}^{-1}\Pi^{-1}\gamma)\\
&-I_\mu(\mathcal R_1\mid p^{-1}\mathcal B({\Sigma_m})\vee\Pi^{-1}\gamma)-I_\mu(\mathcal R_2\mid p^{-1}\mathcal B({\Sigma_m})\vee\mathcal A(\mathcal R_1)\vee \Pi^{-1}\gamma).
\end{align*}
Since $\mathcal R_1$ is $\Pi$-adapted, Lemma~\ref{simple} gives
\[\mathcal A(\mathcal R_1)\vee\tilde{f}^{-1}\Pi^{-1}\gamma=\mathcal A(\mathcal R_1)\vee \Pi^{-1}\gamma.\]
Hence, we obtain
\begin{align*}
I_\mu(\mathcal R_1\vee \mathcal R_2\mid p^{-1}\mathcal B({\Sigma_m})\vee\tilde{f}^{-1}\Pi^{-1}\gamma)-I_\mu(\mathcal R_1\vee \mathcal R_2\mid p^{-1}\mathcal B({\Sigma_m})\vee\Pi^{-1}\gamma)\\
=
I_\mu(\mathcal R_1\mid p^{-1}\mathcal B({\Sigma_m})\vee\tilde{f}^{-1}\Pi^{-1}\gamma)
-I_\mu(\mathcal R_1\mid p^{-1}\mathcal B({\Sigma_m})\vee\Pi^{-1}\gamma).
\end{align*}
The same argument gives
\begin{align*}
I_\mu(\mathcal R_1\vee \mathcal R_2\mid p^{-1}\mathcal B({\Sigma_m})\vee\tilde{f}^{-1}\Pi^{-1}\gamma)-I_\mu(\mathcal R_1\vee \mathcal R_2\mid p^{-1}\mathcal B({\Sigma_m})\vee\Pi^{-1}\gamma)
\\=I_\mu(\mathcal R_2\mid p^{-1}\mathcal B({\Sigma_m})\vee\tilde{f}^{-1}\Pi^{-1}\gamma)-I_\mu(\mathcal R_2\mid p^{-1}\mathcal B({\Sigma_m})\vee\Pi^{-1}\gamma).
\end{align*}
Thus, we obtain
\begin{align*}&I_\mu(\mathcal R_1\mid p^{-1}\mathcal B({\Sigma_m})\vee\tilde{f}^{-1}\Pi^{-1}\gamma)-I_\mu(\mathcal R_1\mid p^{-1}\mathcal B({\Sigma_m})\vee\Pi^{-1}\gamma)
\\&=I_\mu(\mathcal R_2\mid p^{-1}\mathcal B({\Sigma_m})\vee\tilde{f}^{-1}\Pi^{-1}\gamma)-I_\mu(\mathcal R_2\mid p^{-1}\mathcal B({\Sigma_m})\vee\Pi^{-1}\gamma),
\end{align*}which implies that for $\mu$-a.e. $x\in J(\tilde f),$
\[h_{\Pi, p}(\tilde{f},\mu, x, \mathcal R_1)=h_{\Pi, p}(\tilde{f},\mu, x, \mathcal R_2),\] as required.

\subsection{Proof of Theorem A}

Fix a $\Pi$-adapted partition $\mathcal R$.
By Proposition~\ref{noproblem}, the value of the conditional projection entropy does not depend on this
choice. 
Let $p: J(\tilde f)\rightarrow \Sigma_m$ be a measurable map such that
$\sigma \circ p=p\circ \tilde f.$
For an $\tilde f$-invariant Borel probability measure $\mu$, let 
\[\mu=\int_{\Sigma_m} \mu_{\omega}\ d(p_{\ast}\mu)(\omega)\]
be the disintegration of $\mu$ over $p.$
By Remark~\ref{disint}, it is enough to prove the following.
\begin{thm}

Let $\mu$ be an $\tilde f$-invariant Borel probability measure on
$J(\tilde f)$. Then for $\mu$-a.e. $x\in J(\tilde f),$

\begin{equation} 
\label{eqlocal}\lim_{r\to 0}\frac{\log \mu_{x}(B^\Pi(x,r))}{\log r}=
\frac{h_{\Pi, p}(\tilde f,\mu, x)}{E_{\mu}(\log\|\tilde f'(x)\| \mid \mathcal I)}.
\end{equation}
If $\mu$ is ergodic, then for $\mu$-a.e. $x\in J(\tilde f),$
\[
\lim_{r\to 0}\frac{\log \mu_{x}(B^\Pi(x,r))}{\log r}
=
\frac{h_{\Pi, p}(\tilde f,\mu)}{\displaystyle\int_{J(\tilde f)} \log\|\tilde f'(x)\|\,d\mu(x)}.
\]

\end{thm}

\begin{proof}
For any $x\in J(\tilde f)$ and any $n\ge 1$, put 
\[
\rho_0(x):=1, \rho(x):=\|\tilde f'(x)\|^{-1}, \rho_n(x):=\rho(x)\rho(\tilde f x)\cdots \rho(\tilde f^{n-1}x)=\|(\tilde f^n)'(x)\|^{-1}.\]
Since $\tilde f$ is expanding, there exist $C>0$ and $\lambda>1$ such that
$\|(\tilde f^n)'(x)\|\ge C\lambda^n$
for any $x\in J(\tilde f)$ and any $n\ge 1$. 
Hence, for any $x\in J(\tilde f)$ and any $n\ge 1,$
$\rho_n(x)\le C^{-1}\lambda^{-n}.$ 
Fix $c\in(1,\lambda)$. Then
\[M_c:= \sup_{n\ge 0,\ x\in J(\tilde f)} \max\{c^n\rho_n(x),c^{-n}\rho_n(x)\}<\infty .\]
Let $r_0>0$ be the constant given by Lemma~\ref{Ballesti}. 
Choose $0<r_{\ast}<r_0/M_c$. Then for any $n\ge 0$, any $x\in J(\tilde f)$, and both signs $\varepsilon\in\{+1,-1\}$,
$c^{\varepsilon n}\rho_n(x)r_{\ast}<r_0.$
For $\varepsilon\in\{+1,-1\}$, $x\in J(\tilde f)$ and $n\ge 1$, define
\[r_n^\varepsilon(x):=c^{\varepsilon n}\rho_n(x)r_{\ast},\]
\[
H_n^\varepsilon(x):=
\log\frac{\mu_{x}\bigl(B^\Pi(x,r_n^\varepsilon(x))\bigr)}
{\mu_{\tilde f x}\bigl(B^\Pi(\tilde f x,r_{n-1}^\varepsilon(\tilde f x))\bigr)},\]
\[P_n^\varepsilon(x):=
\log\frac{\mu_{x}\bigl(B^\Pi(x,r_n^\varepsilon(x))\cap\mathcal R(x)\bigr)}
{\mu_{x}\bigl(B^\Pi(x,r_n^\varepsilon(x))\bigr)},\]
and
\[Q_n^\varepsilon(x):=
\log\frac{\mu_{x}\bigl(B^{\Pi\circ\tilde f}(x,r_{n-1}^\varepsilon(\tilde f x))\cap\mathcal R(x)\bigr)}
{\mu_{x}\bigl(B^{\Pi\circ\tilde f}( x,r_{n-1}^\varepsilon(\tilde f x))\bigr)}.
\]
We first treat the upper estimate for the local dimension. For this purpose take $\varepsilon=+1$. Applying Lemma~\ref{Ballesti} to $ r=r_{n-1}^+(\tilde f x),$ for any $x\in J(\tilde f)$ and any $n\ge 1$ we have
\[
B^{\Pi\circ\tilde f}(x, r_{n-1}^+(\tilde f x))\cap\mathcal R(x)
\subset
B^\Pi(x,c\rho(x)r_{n-1}^+(\tilde f x))\cap\mathcal R(x)=
B^\Pi(x,r_n^+(x))\cap\mathcal R(x).
\]
By this and Lemma~\ref{fibinv}, for $\mu$-a.e. $x\in J(\tilde f)$ and any $n\ge 1$
\[H_n^+(x)+P_n^+(x)\ge Q_n^+(x).\]
By the definition of $H_n^+$, we have
\[ \log\mu_{x}\bigl(B^\Pi(x,r_n^+(x))\bigr)=
\sum_{j=0}^{n-1}H_{n-j}^+(\tilde f^j x)+\log\mu_{\tilde f^n x}\bigl(B^\Pi(\tilde f^n x,r_{\ast})\bigr).\]
Hence,
\begin{equation}
\label{inethm}
\begin{split}
\frac{1}{n}\log \mu_{x}\bigl(B^\Pi(x,r_n^+(x))\bigr)& \ge \frac{1}{n}\sum_{j=0}^{n-1} \bigl(-P_{n-j}^+(\tilde f^j x)+Q_{n-j}^+(\tilde f^j x)\bigr) \\
&\quad+ \frac{1}{n}\log\mu_{\tilde f^n x}\bigl(B^\Pi(\tilde f^n x,r_{\ast})\bigr).\end{split}
\end{equation}
Applying Lemma~\ref{prodensity} to the measurable maps $\Pi$ and $\Pi\circ\tilde f$, we have
\[P_n^+ \to
-I_\mu(\mathcal R\mid p^{-1}\mathcal B(\Sigma_m)\vee\Pi^{-1}\gamma)\ \text{and}\
Q_n^+ \to
-I_\mu(\mathcal R\mid p^{-1}\mathcal B(\Sigma_m)\vee\tilde f^{-1}\Pi^{-1}\gamma)\ \text{as}\ n\to \infty
\]
almost everywhere and in $L^1$. Thus
\[ -P_n^+ + Q_n^+ \to g\ \text{as}\ n\to \infty\]
almost everywhere and in $L^1$,  where 
\[g:=
I_\mu(\mathcal R\mid p^{-1}\mathcal B(\Sigma_m)\vee\Pi^{-1}\gamma)
-
I_\mu(\mathcal R\mid p^{-1}\mathcal B(\Sigma_m)\vee\tilde f^{-1}\Pi^{-1}\gamma).
\]
Applying Lemma~\ref{Manelem} to the sequence
\[F_n:=-P_n^+ + Q_n^+\]
gives for $\mu$-a.e. $x\in J(\tilde f),$
\begin{align}\label{eqthm1}\lim_{n\to \infty}\frac{1}{n}\sum_{j=0}^{n-1}
\bigl(-P_{n-j}^+(\tilde f^j x)+Q_{n-j}^+(\tilde f^j x)\bigr) =E_{\mu}(g\mid \mathcal I)(x).
\end{align}
Moreover, Lemma~\ref{prozero} gives for $\mu$-a.e. $x\in J(\tilde f)$,
\begin{align}\label{eqthm2}\lim_{n\to \infty}\frac{1}{n}\log\mu_{\tilde f^n x}\bigl(B^\Pi(\tilde f^n x,r_{\ast})\bigr)=0.
\end{align}
Combining \eqref{inethm}, \eqref{eqthm1}, and \eqref{eqthm2} yields for  $\mu$-a.e. $x$,
\begin{align}\label{inethm2}
\liminf_{n\to\infty}
\frac{1}{n}
\log \mu_{x}\bigl(B^\Pi(x,r_n^+(x))\bigr)
\ge E_{\mu}(g\mid \mathcal I)(x).
\end{align}
On the other hand, by Birkhoff's ergodic theorem, for $\mu$-a.e. $x$,
\begin{align}\label{eqthm3}
\frac{1}{n}\log r_n^+(x)=
\log c+
\frac{1}{n}\sum_{k=0}^{n-1}\log\rho(\tilde f^k x)+
\frac{1}{n}\log r_{\ast}\to \log c+ E_{\mu}(\log \rho\mid \mathcal I)(x).
\end{align}
By \eqref{inethm2} and \eqref{eqthm3} for $\mu$-a.e. $x\in J(\tilde f)$,
\[\limsup_{r\to0}\frac{\log\mu_{x}(B^\Pi(x,r))}{\log r}\le \frac{E_{\mu}(-g\mid \mathcal I)(x)}{ E_{\mu}(-\log \rho\mid \mathcal I)(x)-\log c}.
\]
Letting $c\to 1$ gives for $\mu$-a.e. $x\in J(\tilde f)$,
\begin{align}
\label{inelimsup}
\limsup_{r\to0}\frac{\log\mu_{x}(B^\Pi(x,r))}{\log r}\le\frac{h_{\Pi, p}(\tilde{f},\mu, x)}{E_{\mu}(-\log \rho\mid \mathcal I)(x)}.
\end{align}

We next take $\varepsilon=-1$. By Lemma~\ref{Ballesti}, applied with $r=r_{n-1}^-(\tilde f x),$
we have
\[B^\Pi(x,r_n^-(x))\cap\mathcal R(x)=B^\Pi(x,c^{-1}\rho(x)r_{n-1}^-(\tilde f x))\cap\mathcal R(x)\subset B^{\Pi\circ\tilde f}(x,r_{n-1}^-(\tilde f x))\cap\mathcal R(x).
\]
Therefore
\[H_n^-(x)+P_n^-(x)\le Q_n^-(x).\]
Repeating the preceding argument with $\varepsilon = -1$, and using Lemmas~\ref{prodensity}, \ref{prozero}, and ~\ref{Manelem}, we obtain
\[\limsup_{n\to\infty}\frac{1}{n}\log \mu_{x}\bigl(B^\Pi(x,r_n^-(x))\bigr)
\le
-h_{\Pi, p}(\tilde{f},\mu, x).
\]
Moreover, for $\mu$-a.e. $x,$
\[\frac{1}{n}\log r_n^-(x)=-\log c
+\frac{1}{n}\sum_{k=0}^{n-1}\log\rho(\tilde f^k x)+
\frac{1}{n}\log r_{\ast}
\longrightarrow
E_{\mu}(\log \rho\mid \mathcal I)(x)-\log c .\]
Hence, for $\mu$-a.e. $x\in J(\tilde f)$,
\[\liminf_{r\to0}\frac{\log\mu_{x}(B^\Pi(x,r))}{\log r}
\ge
\frac{h_{\Pi, p}(\tilde{f},\mu, x)}{E_{\mu}(-\log \rho\mid \mathcal I)(x)+\log c}.
\]
Letting $c\to 1$ gives for $\mu$-a.e. $x\in J(\tilde f)$,
\begin{align}
\label{ineliminf}
\liminf_{r\to0}\frac{\log\mu_{x}(B^\Pi(x,r))}{\log r}
\ge
\frac{h_{\Pi, p}(\tilde{f},\mu, x)}{E_{\mu}(-\log \rho\mid \mathcal I)(x)}.
\end{align}
Combining \eqref{inelimsup} and \eqref{ineliminf} yields \eqref{eqlocal}.

If $\mu$ is ergodic, then the invariant $\sigma$-algebra $\mathcal I$ is trivial modulo $\mu$.
Hence, for $\mu$-a.e. $x,$
\[
h_{\Pi, p}(\tilde{f},\mu, x)=\int_{J(\tilde f)} h_{\Pi, p}(\tilde{f},\mu, x)\,d\mu=h_{\Pi, p}(\tilde{f},\mu)\] and
\[E_{\mu}(-\log \rho\mid \mathcal I)(x)=\int_{J(\tilde f)}-\log \rho\,d\mu=\int_{J(\tilde f)} \log\|\tilde f^{\prime}(x)\|\,d\mu.\]
Combining this and \eqref{eqlocal} completes the proof.
\end{proof}
\begin{proof}[Proof of Theorem B]
Let $\overline{\omega}\in \Sigma_m$ be a fixed point of $\sigma.$
Define $p:J(\tilde f)\rightarrow \Sigma_m$ by $p(x)=\overline{\omega}$. Then $p^{-1}\mathcal B(\Sigma_m)$ is the
trivial $\sigma$-algebra modulo $\mu$, and the disintegration over $p$
consists of the single measure $\mu$. Hence, Theorem A reduces to
Theorem B.
\end{proof}

\subsection*{Acknowledgments} The author thanks Hiroki Sumi for fruitful discussions and valuable comments.
The author also thanks Takayuki Watanabe for valuable comments and helpful lectures on rational semigroups. This research was supported by the JSPS KAKENHI 25K17282, Grant-in-Aid for Early-Career Scientists.

\end{document}